\documentclass[10pt]{article}
\usepackage{titling}
\usepackage{latexsym}
\usepackage{hyperref}
\usepackage[all]{xy}
\usepackage{xypic}

\usepackage{amsmath}
\usepackage{amssymb}
\usepackage{amsfonts}
\usepackage{amsthm}
\usepackage{mathrsfs}
\usepackage[utf8]{inputenc}
\usepackage{newunicodechar}
\usepackage{stmaryrd}
\usepackage{enumerate}
\usepackage{xcolor}
\usepackage{tikz}
\usetikzlibrary{shapes.geometric}
\usetikzlibrary{calc,arrows.meta}
\usepackage{graphicx}
\usepackage[multiple]{footmisc}

\DeclareUnicodeCharacter{00A0}{ }
\newunicodechar{α}{\ensuremath{\alpha}}
\newunicodechar{β}{\ensuremath{\beta}}
\newunicodechar{χ}{\ensuremath{\chi}}
\newunicodechar{δ}{\ensuremath{\delta}}
\newunicodechar{ε}{\ensuremath{\varepsilon}}
\newunicodechar{Δ}{\ensuremath{\Delta}}
\newunicodechar{η}{\ensuremath{\eta}}
\newunicodechar{γ}{\ensuremath{\gamma}}
\newunicodechar{Γ}{\ensuremath{\Gamma}}
\newunicodechar{ι}{\ensuremath{\iota}}
\newunicodechar{κ}{\ensuremath{\kappa}}
\newunicodechar{λ}{\ensuremath{\lambda}}
\newunicodechar{Λ}{\ensuremath{\Lambda}}
\newunicodechar{μ}{\ensuremath{\mu}}
\newunicodechar{ν}{\ensuremath{\nu}}
\newunicodechar{ω}{\ensuremath{\omega}}
\newunicodechar{Ω}{\ensuremath{\Omega}}
\newunicodechar{π}{\ensuremath{\pi}}
\newunicodechar{φ}{\ensuremath{\phi}}
\newunicodechar{Φ}{\ensuremath{\Phi}}
\newunicodechar{ψ}{\ensuremath{\psi}}
\newunicodechar{Ψ}{\ensuremath{\Psi}}
\newunicodechar{ρ}{\ensuremath{\rho}}
\newunicodechar{σ}{\ensuremath{\sigma}}
\newunicodechar{Σ}{\ensuremath{\Sigma}}
\newunicodechar{τ}{\ensuremath{\tau}}
\newunicodechar{θ}{\ensuremath{\theta}}
\newunicodechar{Θ}{\ensuremath{\Theta}}
\newunicodechar{ξ}{\ensuremath{\xi}}
\newunicodechar{ζ}{\ensuremath{\zeta}}

\newunicodechar{∞}{\ensuremath{\infty}}
\newunicodechar{Σ}{\ensuremath{\Sigma}}

\DeclareFontFamily{U}{wncy}{}
\DeclareFontShape{U}{wncy}{m}{n}{<->wncyr10}{}
\DeclareSymbolFont{mcy}{U}{wncy}{m}{n}
\DeclareMathSymbol{\Sh}{\mathord}{mcy}{"58}

\newtheorem{theo}{Theorem}
\newtheorem{prop}[theo]{Proposition}

\newtheorem{lemm}[theo]{Lemma}

\theoremstyle{definition}

\newtheorem{rema}[theo]{Remark}

\newtheorem{ques}[theo]{Question}

\newcommand{\cE}{\mathcal{E}}
\newcommand{\cG}{\mathcal{G}}
\newcommand{\GG}{\mathbb{G}}

\newcommand{\NN}{\mathbb{N}}
\newcommand{\OO}{\mathcal{O}}
\newcommand{\p}{\mathfrak{p}}
\newcommand{\PP}{\mathbb{P}}

\newcommand{\RR}{\mathbb{R}}
\newcommand{\cV}{\mathcal{V}}

\newcommand{\ZZ}{\mathbb{Z}}

\DeclareMathOperator{\Tor}{Tor}

\DeclareMathOperator{\coker}{coker}

\DeclareMathOperator{\Pic}{Pic}

\DeclareMathOperator{\Spec}{Spec}
\DeclareMathOperator{\Proj}{Proj}
\DeclareMathOperator{\Bl}{Bl}

\DeclareMathOperator{\trop}{trop}
\DeclareMathOperator{\relint}{relint}
\DeclareMathOperator{\Kdim}{Krull.dim}

\DeclareMathOperator{\Fan}{Fan}
\DeclareMathOperator{\Krulldim}{Krull\ dim.}
\DeclareMathOperator{\valdim}{val.dim.}
\DeclareMathOperator{\cdh}{cdh}

\newcommand{\length}{{\operatorname{length}}}
\newcommand{\xra}[1]{\xrightarrow{#1}}

\newcommand{\colim}{\varinjlim}

\newcommand{\sing}{\mathrm{sing}}

\definecolor{readableblue}{HTML}{3399FF}

\renewcommand{\hom}{\mathrm{hom}}

\usetikzlibrary{arrows.meta}

\title{
Some explicit counter-examples to \mbox{Weibel's conjecture}}
\author{Shane Kelly}

\begin{document}

\maketitle

\begin{abstract}
We give two related but different methods for constructing rings $R$ of Krull dimension 1 and $K_{-d}(R) \neq 0$. The first works for $d = 2$, and the second works for any $d \geq 2$.
\end{abstract}

In \cite{Wei80}%
\ Weibel asked the following question.

\begin{ques}[{\cite[Questions 2.9]{Wei80}%
}] \label{ques1}
If $R$ is a commutative noetherian ring of Krull dimension $d$, is $K_*(R) = 0$ for $* < -d$?
\end{ques}

This stimulated four decades of research, 
\cite[Thm.1.2]{Wei89}%
, 
\cite[Thm.4.4]{Wei01}%
, 
\cite[Thm.8.14]{Hae04}%
, 
\cite[Thm.6.2]{CHSW08}%
, 
\cite[Thm.1.3]{Kri08}%
, 
\cite[Thm.A]{GH10}%
, 
\cite[Thm.3.11]{Cis12}%
, 
\cite[Thm.3.5]{Kel13}%
, 
\cite[Thm.4.7]{Mor16}%
, 
\cite[Thm.1, Cor.2]{KS16}%
. The Kerz--Strunk--Tamme article 
\cite[Thm.B]{KST18}%
\ 
is generally considered the definitive answer. It is notable for resolving the question as stated, and being one of the first strict applications of derived algebraic geometry to answer a question about non-derived rings. None-the-less, research into this question has not abated, see for example
\cite[Thm.1.1]{HK19}%
, 
\cite[Thm.1.1]{Sad19}%
, 
\cite[Thm.0.1]{KS19}%
, 
\cite[Thm.1.1]{Sta20}%
, 
\cite[Thm.4.1]{Kha21}%
, 
\cite[Thm.D]{BKRS22}%
, 
\cite[Prop.4.5]{MM22}%
. 

In particular, for nonnoetherian rings the answer to Question~\ref{ques1} remains: yes, if dimension is replaced with \emph{valuative dimension} \cite[Rem.3.4]{KM18}%
,
\cite[Thm.2.4.15]{EHIK21}%
,
\cite[Thm.6.5]{KST25}%
. In short, the valuative dimension of an integral ring is the Krull dimension of its Riemann-Zariski space. Equivalently, it is the supremum of the Krull dimensions over all modifications,
\cite[Defs.2.2.1--2.2.2, Prop.2.3.8]{EHIK21},%
\[ \valdim A = \sup_{\substack{Y \to \Spec A \\ \textrm{proper}\\\textrm{birational}}} \Krulldim Y. \]
 When $R$ is noetherian, $\Krulldim R = \valdim R$ but in general $\Krulldim R \leq \valdim R$, 
\cite[Prop.2.3.2(8),(9)]{EHIK21}%
.

In this note we present the first known counter-examples to the non-noetherian version of Weibel's question.

\begin{theo}
For all $d \geq 2$, there exists a ring $R$ such that $\Spec R$ has exactly two points, and $K_{-d}(R) \cong \ZZ$.
\end{theo}

\begin{rema} \label{rema:dequals1v2}
Below $d = 2$, namely a ring $R$ with $\Krulldim R = 0$ and $K_{-1}(R) \neq 0$ is well-known to be impossible via standard techniques. Indeed, negative $K$-theory is nilinvariant, \cite[III.2, Props.2.10, 2.12; XII.7, p.664]{Bas68}%
, so we can assume $R$ is reduced. A reduced ring of Krull dimension zero is absolutely flat, \cite[092F]{stacks-project}%
. For every $d\geq0$, the ring $R[T_1,\ldots,T_d]$ is coherent \cite[Thm.1]{Sab74} %
and every finitely presented $R[T_1,\ldots,T_d]$-module has finite projective dimension \cite[068X, 0EWZ, 05CX, 00OQ]{stacks-project}%
. Thus $R$ is regular and stably coherent, and hence has no negative $K$-theory, \cite[proof of Thm.3.3]{KM18}%
.

On the other hand, Christopher La Fond has some interesting work in progress about building schemes of dimension zero with non-vanishing negative $K$-theory in the non-qcqs setting.
\end{rema}

\emph{Some topology.} The idea for the construction is an algebraic version of the construction of the $d$-sphere (a CW complex with one $0$-cell and one $d$-cell) as a $d$-ball modulo its boundary $S^d = D^d / \partial D^d$. Model independent homotopy theorists will mentally replace $D^d$ with a point and model category theorists may replace $D^d$ with an arbitrary contractible space equipped with an embedding of $S^{d-1}$ such as $S^{d-1} \subseteq \RR^d$. 

\emph{Sketch of the construction.} Our $K_{<0}$-contractible space $\Spec R_\infty$ is a semilocalisation of a Riemann-Zariski space. In general these have no negative $K$-theory,%
\footnote{Their local rings are valuation rings, which have no negative $K$-groups, \cite[Thm.1.3(iii)]{KM18}%
. It follows that their structure sheaves are coherent: for a nonzero map $f=(a_1,\ldots,a_n):\mathcal O^n\to\mathcal O$, the valuation property allows us, locally, to choose $a_i$ dividing every $a_j$. Write $a_j=a_i b_j$ with $b_i=1$. Multiplication by $a_i$ is injective because $\mathcal O$ is a subsheaf of the constant function-field sheaf, while $(b_1,\ldots,b_n)$ is split surjective. Thus $\ker(f)$ is finite locally free.} %
but we will not need this result as we explicitly construct it as a filtered limit of regular schemes $R_∞ = \colim R_λ$. 

Our embedded $S^{d-1}$ is a configuration of hypersurfaces forming a $(d{-}1)$-dimensional polyhedron with infinitely many faces. As with $R_∞$, it is constructed as a limit $R_∞/I_∞ = \colim R_λ / I_λ$. The reader is invited to picture a polyhedron, iteratively replace each corner with a smaller face, and observe that limit will resemble a sphere. %

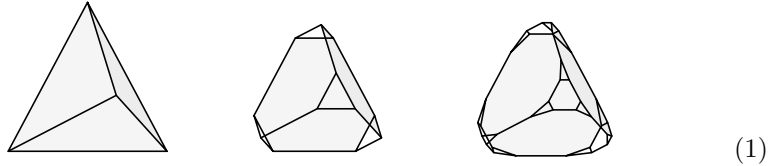
\begin{equation} \label{equa:corners}
\begin{tikzpicture}[
  line width=.6pt,
  line cap=round,
  line join=round,
  >={Stealth[length=5pt]}
]
\begin{scope}[shift={(-3,0)}]
  \coordinate (PT) at (0,1.25);
  \coordinate (PA) at (-1.05,-.72);
  \coordinate (PB) at (1.05,-.72);
  \coordinate (PC) at (.38,.02);

  \path[fill=black!4] (PT)--(PA)--(PB)--cycle;
  \path[fill=black!7] (PT)--(PB)--(PC)--cycle;
  \draw (PT)--(PA)--(PB)--(PT)--(PC)--(PB);
  \draw (PA)--(PC);
\end{scope}

\begin{scope}
  \coordinate (MT) at (0,1.25);
  \coordinate (MA) at (-1.05,-.72);
  \coordinate (MB) at (1.05,-.72);
  \coordinate (MC) at (.38,.02);
  \def\cut{.24}

  \coordinate (MtA) at ($(MT)!\cut!(MA)$);
  \coordinate (MtB) at ($(MT)!\cut!(MB)$);
  \coordinate (MtC) at ($(MT)!\cut!(MC)$);
  \coordinate (MaT) at ($(MA)!\cut!(MT)$);
  \coordinate (MaB) at ($(MA)!\cut!(MB)$);
  \coordinate (MaC) at ($(MA)!\cut!(MC)$);
  \coordinate (MbT) at ($(MB)!\cut!(MT)$);
  \coordinate (MbA) at ($(MB)!\cut!(MA)$);
  \coordinate (MbC) at ($(MB)!\cut!(MC)$);
  \coordinate (McT) at ($(MC)!\cut!(MT)$);
  \coordinate (McA) at ($(MC)!\cut!(MA)$);
  \coordinate (McB) at ($(MC)!\cut!(MB)$);

  \path[fill=black!4]
    (MtA)--(MaT)--(MaB)--(MbA)--(MbT)--(MtB)--cycle;
  \path[fill=black!7]
    (MtB)--(MbT)--(MbC)--(McB)--(McT)--(MtC)--cycle;

  \draw
    (MtA)--(MtB)--(MtC)--cycle
    (MaT)--(MaB)--(MaC)--cycle
    (MbT)--(MbA)--(MbC)--cycle
    (McT)--(McA)--(McB)--cycle;
  \draw
    (MtA)--(MaT)
    (MtB)--(MbT)
    (MtC)--(McT)
    (MaB)--(MbA)
    (MaC)--(McA)
    (MbC)--(McB);
\end{scope}

\begin{scope}[shift={(3,0)}]
  \coordinate (NT) at (0,1.34);
  \coordinate (NA) at (-1.13,-.78);
  \coordinate (NB) at (1.13,-.78);
  \coordinate (NC) at (.41,.02);
  \def\firstcut{.24}
  \def\secondcut{.22}

  \coordinate (N0)  at ($(NT)!\firstcut!(NA)$);
  \coordinate (N1)  at ($(NT)!\firstcut!(NB)$);
  \coordinate (N2)  at ($(NT)!\firstcut!(NC)$);
  \coordinate (N3)  at ($(NA)!\firstcut!(NT)$);
  \coordinate (N4)  at ($(NA)!\firstcut!(NB)$);
  \coordinate (N5)  at ($(NA)!\firstcut!(NC)$);
  \coordinate (N6)  at ($(NB)!\firstcut!(NT)$);
  \coordinate (N7)  at ($(NB)!\firstcut!(NA)$);
  \coordinate (N8)  at ($(NB)!\firstcut!(NC)$);
  \coordinate (N9)  at ($(NC)!\firstcut!(NT)$);
  \coordinate (N10) at ($(NC)!\firstcut!(NA)$);
  \coordinate (N11) at ($(NC)!\firstcut!(NB)$);

  \coordinate (N0x1) at ($(N0)!\secondcut!(N1)$);
  \coordinate (N0x2) at ($(N0)!\secondcut!(N2)$);
  \coordinate (N0x3) at ($(N0)!\secondcut!(N3)$);
  \coordinate (N1x0) at ($(N1)!\secondcut!(N0)$);
  \coordinate (N1x2) at ($(N1)!\secondcut!(N2)$);
  \coordinate (N1x6) at ($(N1)!\secondcut!(N6)$);
  \coordinate (N2x0) at ($(N2)!\secondcut!(N0)$);
  \coordinate (N2x1) at ($(N2)!\secondcut!(N1)$);
  \coordinate (N2x9) at ($(N2)!\secondcut!(N9)$);
  \coordinate (N3x0) at ($(N3)!\secondcut!(N0)$);
  \coordinate (N3x4) at ($(N3)!\secondcut!(N4)$);
  \coordinate (N3x5) at ($(N3)!\secondcut!(N5)$);
  \coordinate (N4x3) at ($(N4)!\secondcut!(N3)$);
  \coordinate (N4x5) at ($(N4)!\secondcut!(N5)$);
  \coordinate (N4x7) at ($(N4)!\secondcut!(N7)$);
  \coordinate (N5x3) at ($(N5)!\secondcut!(N3)$);
  \coordinate (N5x4) at ($(N5)!\secondcut!(N4)$);
  \coordinate (N5x10) at ($(N5)!\secondcut!(N10)$);
  \coordinate (N6x1) at ($(N6)!\secondcut!(N1)$);
  \coordinate (N6x7) at ($(N6)!\secondcut!(N7)$);
  \coordinate (N6x8) at ($(N6)!\secondcut!(N8)$);
  \coordinate (N7x4) at ($(N7)!\secondcut!(N4)$);
  \coordinate (N7x6) at ($(N7)!\secondcut!(N6)$);
  \coordinate (N7x8) at ($(N7)!\secondcut!(N8)$);
  \coordinate (N8x6) at ($(N8)!\secondcut!(N6)$);
  \coordinate (N8x7) at ($(N8)!\secondcut!(N7)$);
  \coordinate (N8x11) at ($(N8)!\secondcut!(N11)$);
  \coordinate (N9x2) at ($(N9)!\secondcut!(N2)$);
  \coordinate (N9x10) at ($(N9)!\secondcut!(N10)$);
  \coordinate (N9x11) at ($(N9)!\secondcut!(N11)$);
  \coordinate (N10x5) at ($(N10)!\secondcut!(N5)$);
  \coordinate (N10x9) at ($(N10)!\secondcut!(N9)$);
  \coordinate (N10x11) at ($(N10)!\secondcut!(N11)$);
  \coordinate (N11x8) at ($(N11)!\secondcut!(N8)$);
  \coordinate (N11x9) at ($(N11)!\secondcut!(N9)$);
  \coordinate (N11x10) at ($(N11)!\secondcut!(N10)$);

  \path[fill=black!4]
    (N0x1)--(N0x3)--(N3x0)--(N3x4)--
    (N4x3)--(N4x7)--(N7x4)--(N7x6)--
    (N6x7)--(N6x1)--(N1x6)--(N1x0)--cycle;
  \path[fill=black!7]
    (N1x2)--(N1x6)--(N6x1)--(N6x8)--
    (N8x6)--(N8x11)--(N11x8)--(N11x9)--
    (N9x11)--(N9x2)--(N2x9)--(N2x1)--cycle;

  \draw
    (N0x1)--(N0x2)--(N0x3)--cycle
    (N1x0)--(N1x2)--(N1x6)--cycle
    (N2x0)--(N2x1)--(N2x9)--cycle
    (N3x0)--(N3x4)--(N3x5)--cycle
    (N4x3)--(N4x5)--(N4x7)--cycle
    (N5x3)--(N5x4)--(N5x10)--cycle
    (N6x1)--(N6x7)--(N6x8)--cycle
    (N7x4)--(N7x6)--(N7x8)--cycle
    (N8x6)--(N8x7)--(N8x11)--cycle
    (N9x2)--(N9x10)--(N9x11)--cycle
    (N10x5)--(N10x9)--(N10x11)--cycle
    (N11x8)--(N11x9)--(N11x10)--cycle;

  \draw
    (N0x1)--(N1x0)
    (N1x2)--(N2x1)
    (N2x0)--(N0x2)
    (N3x4)--(N4x3)
    (N4x5)--(N5x4)
    (N5x3)--(N3x5)
    (N6x7)--(N7x6)
    (N7x8)--(N8x7)
    (N8x6)--(N6x8)
    (N9x10)--(N10x9)
    (N10x11)--(N11x10)
    (N11x9)--(N9x11)
    (N0x3)--(N3x0)
    (N1x6)--(N6x1)
    (N2x9)--(N9x2)
    (N4x7)--(N7x4)
    (N5x10)--(N10x5)
    (N8x11)--(N11x8);
\end{scope}
\end{tikzpicture}
\end{equation}

In our actual construction, the corner replacements are blowups with centre inside a sncd (the polyhedron),%
\footnote{ 
In fact, picturing each successive face as smaller than the previous ones is mathematically justified by work of Yasuda \cite[Cor.4.10, Prop.4.11, Thm.7.4, §8]{Yas17}, %
extending mass formulas %
of Serre, \cite[Thm.2, p.13]{Ser82}%
,  Bhargava, \cite[Thm.1.1]{Bha07}%
, and Kedlaya, \cite[Thm.8.5]{Ked07} %
(recovered for $p \neq 2$; see \cite[§5]{WY17}%
). %
For motivic refinements of the Krasner--Serre--Bhargava formulas, see \cite[Thm.1.1]{Yas24}%
.
} %
 and one actually has to blowup all face intersections, not just those of dimension zero.

The purpose of all these blowups is to slide all irrelevant subvarieties off the corners. Then taking the semi-localisation at these corners removes them entirely. So the only non-$\Spec R_∞/I_∞$ point of $\Spec R_∞$ is the generic point.

Experts will be unsurprised that the analogue of the quotient $S^d = D^d / \partial D^d = D^d \sqcup_{\partial D^d} \ast$ is achieved using a Milnor square $R = R_∞ \times_{R_∞ / I_∞} k$. The $K_{1-d}$ of our boundary $\Spec R_∞ / I_∞$ is calculated using the identification $K_{-\dim X}(X) \cong H_{\cdh}^{\dim X}(X, \ZZ)$, \cite[Cor.D]{KST18}%
, and the identification of $H_{\cdh}^{d-1}(\Spec R_∞ / I_∞, \ZZ)$ with the \v{C}ech cohomology of the (closed) covering of $\Spec R_∞ / I_∞$ by its irreducible components. That is, $H_{\cdh}^{d-1}(\Spec R_∞ / I_∞, \ZZ)$ is calculated by the same combinatorics as $H_{\sing}^{d-1}(-, \ZZ)$ of the corresponding $(d{-}1)$-dimensional polyhedron.


\emph{Computer assistance.} 
The idea to use Milnor squares was clear to everyone (the author knows of) who was looking for these examples. The author was originally unsuccessfully trying to glue valuation rings together in interesting configurations. 
The technique of using (finite) polygons to produce negative $K$-theory classes is also well-known to everyone. The calculation for a finite polygon is already implicit in Bass's conductor formula, \cite[XII.10.4(c)]{Bas68}%
. %
The infinite polygon construction surfaced as part of an unsuccessful attempt at a counterexample by a computer assistant;%
\footnote{
Mathematicians are currently being instrumentalised in a large scale advertising campaign by AI companies competing for market monopoly. %
Public grant money is being paid to companies which then receive free advertising from us, %
often while showing little regard for our research priorities. %
More seriously, as part of this advertising campaign, %
some companies are pushing misinformation about the goals of mathematical research. %
This misinformation has the real potential to influence government funding decisions in %
a way that funnels money away from scientific goals, and towards private interests. 

This advertising campaign has more serious collateral damage. %
Mathematics depends on a human research ecosystem which continually generates new problems %
and sustains the community capable of recognising and pursuing them. %
Human generated research problems are a scarce and essential part of this %
mathematical ecosystem. %
Conceptually incoherent and empirically unsupported rhetoric portraying mathematicians %
as replaceable threatens a pipeline of early career researchers already in precarious employment conditions. %
Damaging either undermines the infrastructure on which future mathematical progress depends.

While the author \emph{does} want to be transparent about the use of computer assistance in the preparation of the current manuscript, he does not want to participate in this advertising campaign and contribute to this destruction of the commons. As such, the model and company name will remain absent from this manuscript, but are available upon request.
} %
%
 the computer assistant found it in 
\cite[Ex.5.19]{Aok24} %
 %
where it is attributed to Lazard 
\cite[Prop.7.2]{Laz67}%
. %
A successful counter-example construction for $d = 2$, similar in spirit to the one in Section~\ref{sec:dis2} but using explicit polynomials was proposed by a computer assistant but only understood by the author after he had constructed his own version. The toric varieties construction for $d \geq 2$ is the author's, but the observation that you can track intersections using tropical varieties came from a computer assistant. Most of the technical toric work was done in collaboration with a computer assistant in a long conversation involving multiple prompts.\footnote{Prompts largely of the form ``find a reference for...'' or ``check the attached for errors''. } Chat history is available upon request.

\emph{Acknowledgement.} The author thanks Marc Hoyois, Christopher La Fond, Matthew Morrow, Shuji Saito, Georg Tamme for their interest in this question, Michael McBreen for stimulating conversations while the example was being written up, and Sebastian Szturo for many interesting conversations about artificial intelligence.

\section{A construction for $d = 2$} \label{sec:dis2}

\begin{prop}
There exists a ring $R$ with
\[ \Kdim R = 1, \qquad \textrm{ and } \qquad K_{-2}(R) \cong \ZZ. \]
\end{prop}

\begin{proof}
Let $X_0$ be a smooth surface over a field $k$ and  $E_0 = E_{0,0} \cup E_{0,1} \cup E_{0,2}$ a snc divisor arranged in a triangle. Let $V_0 = \{ E_{0,0} \cap E_{0,1}, E_{0,0} \cap E_{0,2}, E_{0,1} \cap E_{0,2} \}$ be the set of intersections of branches and for simplicity suppose all $v \in V_0$ have $k = k(v)$. For example we could take 
\begin{align*}
X_0 &= \PP^2 = \Proj k[x, y, z], \\
E_0 &= \Proj \tfrac{k[x, y, z]}{xyz} = E_{0,x} \cup E_{0,y} \cup E_{0,z}, \\
V_0 &= \{v_{0,xy}, v_{0,xz}, v_{0,yz}\}.
\end{align*}
For $n > 0$, inductively perform the blowup of $X_n$ in $V_n$, set $E_{n+1}$ to be the reduced pullback of $E_n$ and let $V_{n+1}$ be the new set of intersections.
\begin{align*}
X_{n} &= \Bl_{V_{n{-}1}} (X_{n{-}1}) \xra{π_n} X_{n-1}, \\
E_{n} &= (π^{-1}(E_{n-1}))_{red} = \bigcup_{i \in I_{n}} E_{n, i}, \\
V_{n} &= \bigcup_{i,j \in I} E_{n,i} \cap E_{n,j}.
\end{align*}

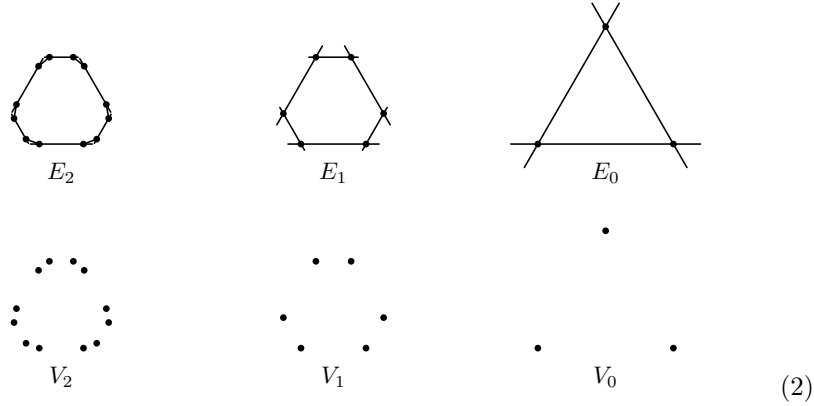
\begin{equation} \label{equa:corner-blowups}
\begin{tikzpicture}%
[line width=.6pt,line cap=round,
scale=.9,transform shape]
\def\radius{1.15}
\def\firstcut{.26}
\def\secondcut{.16}
\def\overhang{.20}

\begin{scope}[shift={(8,3)}]
  \coordinate (e0a) at (90:\radius);
  \coordinate (e0b) at (210:\radius);
  \coordinate (e0c) at (330:\radius);
\end{scope}

\begin{scope}[shift={(4,3)}]
  \coordinate (A) at (90:\radius);
  \coordinate (B) at (210:\radius);
  \coordinate (C) at (330:\radius);

  \coordinate (e1q0) at ($(A)!\firstcut!(C)$);
  \coordinate (e1q1) at ($(A)!\firstcut!(B)$);
  \coordinate (e1q2) at ($(B)!\firstcut!(A)$);
  \coordinate (e1q3) at ($(B)!\firstcut!(C)$);
  \coordinate (e1q4) at ($(C)!\firstcut!(B)$);
  \coordinate (e1q5) at ($(C)!\firstcut!(A)$);
\end{scope}

\begin{scope}[shift={(0,3)}]
  \coordinate (A) at (90:\radius);
  \coordinate (B) at (210:\radius);
  \coordinate (C) at (330:\radius);

  \coordinate (q0) at ($(A)!\firstcut!(C)$);
  \coordinate (q1) at ($(A)!\firstcut!(B)$);
  \coordinate (q2) at ($(B)!\firstcut!(A)$);
  \coordinate (q3) at ($(B)!\firstcut!(C)$);
  \coordinate (q4) at ($(C)!\firstcut!(B)$);
  \coordinate (q5) at ($(C)!\firstcut!(A)$);

  \coordinate (e2r0)  at ($(q0)!\secondcut!(q5)$);
  \coordinate (e2r1)  at ($(q0)!\secondcut!(q1)$);
  \coordinate (e2r2)  at ($(q1)!\secondcut!(q0)$);
  \coordinate (e2r3)  at ($(q1)!\secondcut!(q2)$);
  \coordinate (e2r4)  at ($(q2)!\secondcut!(q1)$);
  \coordinate (e2r5)  at ($(q2)!\secondcut!(q3)$);
  \coordinate (e2r6)  at ($(q3)!\secondcut!(q2)$);
  \coordinate (e2r7)  at ($(q3)!\secondcut!(q4)$);
  \coordinate (e2r8)  at ($(q4)!\secondcut!(q3)$);
  \coordinate (e2r9)  at ($(q4)!\secondcut!(q5)$);
  \coordinate (e2r10) at ($(q5)!\secondcut!(q4)$);
  \coordinate (e2r11) at ($(q5)!\secondcut!(q0)$);
\end{scope}

\foreach \u/\v in {e0a/e0b,e0b/e0c,e0c/e0a}
  \draw ($(\u)!-\overhang!(\v)$)--($(\v)!-\overhang!(\u)$);

\foreach \i/\j in {0/1,1/2,2/3,3/4,4/5,5/0}
  \draw ($(e1q\i)!-\overhang!(e1q\j)$)--
        ($(e1q\j)!-\overhang!(e1q\i)$);

\foreach \i/\j in {
  0/1,1/2,2/3,3/4,4/5,5/6,
  6/7,7/8,8/9,9/10,10/11,11/0}
  \draw ($(e2r\i)!-\overhang!(e2r\j)$)--
        ($(e2r\j)!-\overhang!(e2r\i)$);

\foreach \p in {e0a,e0b,e0c}{
  \fill (\p) circle (1.4pt);
  \fill ($(\p)+(0,-3)$) circle (1.4pt);
}

\foreach \p in {e1q0,e1q1,e1q2,e1q3,e1q4,e1q5}{
  \fill (\p) circle (1.4pt);
  \fill ($(\p)+(0,-3)$) circle (1.4pt);
}

\foreach \p in {
  e2r0,e2r1,e2r2,e2r3,e2r4,e2r5,
  e2r6,e2r7,e2r8,e2r9,e2r10,e2r11}{
  \fill (\p) circle (1.4pt);
  \fill ($(\p)+(0,-3)$) circle (1.4pt);
}

\node at (0,2) {$E_2$};
\node at (4,2) {$E_1$};
\node at (8,2) {$E_0$};

\node at (0,-1) {$V_2$};
\node at (4,-1) {$V_1$};
\node at (8,-1) {$V_0$};
\end{tikzpicture}
\end{equation}

For $n \geq 0$ let $R_n$ be the semilocal ring at $V_n \subseteq X_n$ and let $I_n \subseteq R_n$ be the ideal corresponding to $E_n \cap \Spec R_n$.
\[ R_n = \bigcap_{v \in V_n} \OO_{X_n, v} \subseteq k(X_n), \qquad V(I_n) = E_n \cap \Spec R_n. \]
In particular, we have sequences of ring homomorphisms
\[ R_0 \to R_1 \to R_2 \to \dots \to R_∞ = \colim R_n \]
\[ R_0/I_0 \to R_1/I_1 \to R_2 / I_2 \to \dots \to R_∞/I_∞ = \colim R_n/I_n \]
Set
\[ R = R_∞ \times_{R_∞/I∞} k. \]

\emph{The Krull dimension of $R$ is one.} Since the topological space of the pushout scheme $\Spec R_∞ \sqcup_{\Spec (R_∞ / I_∞)} \Spec k$ is the pushout topological space, it suffices to show that every nongeneric point $y$ of $\Spec R_∞$ lies in $\Spec R_∞ / I_∞$. Equivalently, we want $y_n \in \Spec R_n / I_n$ for all $n$ where $y_n$ is the image of $y$. %
Considering the canonical cartesian square of inclusions
\[ \xymatrix{
\Spec R_n/I_n \ar[r]_-{\subseteq} \ar[d]^{\cap|} & E_n \ar[d]^{\cap|} \\
\Spec R_n \ar[r]_{\subseteq} & X_n
} \] 
one more way to say this is that the irreducible subvariety $Y_n = \overline{\{y_n\}} \subseteq X_n$ is contained in $E_n$ for all $n$. Note that since $y_n \in \Spec R_n = \bigcup_{v \in V_n} \Spec \OO_{X_n, v}$, we necessarily have $Y_n \cap V_n \neq \varnothing$.

Suppose not. Then we have a sequence of proper irreducible subvarieties $\dots \to Y_2 \to Y_1 \to Y_0$ such that $Y_n \cap V_n \neq \varnothing$ and $Y_n \not \subseteq E_n$ for all $n$. In particular, each $Y_n$ is an irreducible curve. Since we are successively blowing up $V_n$, for some $n_0$ the curves $Y_n$ are smooth at $V_n$ for all $n \geq n_0$, \cite[0BI4, 0BXQ]{stacks-project}%
. Then we are in the situation of Lemma~\ref{lemm:three-curves}. In particular, continuing to blow up a finite number of times we eventually have $Y_N \cap V_N = \varnothing$, contradicting the assumption $y_N \in \Spec R_N$.

\emph{The $K$-group $K_{-2}(R)$ is $\ZZ$.} 
From the exact sequence 
\[ \underbrace{K_{-1}(R_∞) \oplus K_{-1}(k)}_{=0} \to K_{-1}(R_∞/I_∞) \to K_{-2}(R) \to \underbrace{K_{-2}(R_∞) \oplus K_{-2}(k)}_{=0} \]
it suffices to show that $K_{-1}(R_∞/I_∞) \cong \ZZ$. Here, the vanishing comes from regularity. Namely, $R_∞$ is a filtered colimit of a sequence of semilocalisations of smooth varieties $X_n$ and  regular rings have no negative $K$-theory. For the claim about $K_{-1}(R_∞/I_∞)$ notice that since the branches of $E_n$ all intersect transversally, $\Spec R_n/I_n$ can be constructed as the pushout
\[ \xymatrix{
\underset{{i \in \ZZ/3\cdot 2^n}}{\bigsqcup} 
(C_{n,i} \cap V_n) \ar[d] \ar[r] & 
\underset{{i \in \ZZ/3\cdot 2^n}}{\bigsqcup} 
C_{n,i} \ar[d] \\
\underset{{i \in \ZZ/3\cdot 2^n}}{\bigsqcup} \Spec k \ar[r] & \Spec R_n / I_n
} \]
where $C_{n,i} = E_{n,i} \cap \Spec R_n/I_n$. 
\[
\begin{tikzpicture}[>=Stealth]

  \node[minimum width=2.6cm,minimum height=2.8cm] (B) at (4.5,3.3) {};
  \node[minimum width=2.6cm,minimum height=2.8cm] (C) at (0,0) {};
  \node[minimum width=2.6cm,minimum height=2.8cm] (D) at (4.5,0) {};

\node[minimum width=2.6cm,minimum height=2.8cm] (A) at (0,3.3) {};

\begin{scope}[shift={(A.center)}]
  \foreach \k in {0,...,5}
    \coordinate (a\k) at ({cos(60*\k)},{sin(60*\k)});

  \foreach \k in {0,...,5}{
    \pgfmathtruncatemacro{\j}{mod(\k+1,6)}
    \fill ($(a\k)!.13!(a\j)$) circle (1.5pt);
    \fill ($(a\k)!.87!(a\j)$) circle (1.5pt);
  }

\end{scope}

  \begin{scope}[shift={(B.center)}]
    \foreach \k in {0,...,5}
      \coordinate (b\k) at ({cos(60*\k)},{sin(60*\k)});

    \foreach \k in {0,...,5}{
      \pgfmathtruncatemacro{\j}{mod(\k+1,6)}
      \draw ($(b\k)!.13!(b\j)$) -- ($(b\k)!.87!(b\j)$);
      \fill ($(b\k)!.13!(b\j)$) circle (1.2pt);
      \fill ($(b\k)!.87!(b\j)$) circle (1.2pt);
    }

  \end{scope}

  \begin{scope}[shift={(C.center)}]
    \foreach \k in {0,...,5}
      \fill ({cos(60*\k)},{sin(60*\k)}) circle (1.5pt);
  \end{scope}

  \begin{scope}[shift={(D.center)}]
    \draw plot[domain=0:360,samples=7]
      ({cos(\x)},{sin(\x)});
  \end{scope}

  \draw[->] (A.east)  -- (B.west);
  \draw[->] (A.south) -- (C.north);
  \draw[->] (B.south) -- (D.north);
  \draw[->] (C.east)  -- (D.west);
\end{tikzpicture}
\]
Since $k$ and the $C_{n,i}$ are regular (the latter are semilocalisations of the smooth curves $E_{n,i}$) they have no negative $K$-theory so 
\begin{align*}
&K_{-1}(R_n/I_n) \\
&= \coker\left ( \left ( \bigoplus_{i \in \ZZ/3 {\cdot} 2^n} K_0(k) \right ) \oplus \left ( \bigoplus_{i \in \ZZ/3 {\cdot} 2^n} K_0(C_{n,i}) \right ) \to \bigoplus_{i \in \ZZ/3 {\cdot} 2^n} K_0(C_{n,i} \cap V_n) \right )
\end{align*}
Since the $C_{n,i}$ are (semilocalisations of) smooth irreducible curves, we have $K_0 = \ZZ \oplus \Pic$. The $\Pic$ part is killed in the $K_0$ of the points $v \in V_n$ so it doesn't contribute to the cokernel and we get 
\[ K_{-1}(R_n/I_n) = \coker\left ( \left ( \bigoplus_{i \in \ZZ/3 {\cdot} 2^n} \ZZ  \right ) \oplus \left ( \bigoplus_{i \in \ZZ/3 {\cdot} 2^n} \ZZ \right ) \to \bigoplus_{i \in \ZZ/3 {\cdot} 2^n} (\ZZ \oplus \ZZ)  \right ) \]
The combinatorics of this cokernel are the same combinatorics calculating $H^1$ of a $3 {\cdot} 2^n$-gon, \cite[XII.10.4(c)]{Bas68}%
. Hence, $\ZZ$. Moreover, the transition morphisms 
\[ \Spec R_{n+1}/I_{n+1} \to \Spec R_n / I_n \]
correspond to collapsing each alternative edge to a point. Hence, this polygon description shows that they induce isomorphisms
\[ K_{-1}(R_{n+1}/I_{n+1}) \stackrel{\sim}{\leftarrow} K_{-1}(R_{n}/I_{n}) \]
Taking the colimit gives
\[ K_{-2}(R) = K_{-1}(R_∞/I_∞) = \colim K_{-1}(R_n/I_n) = \colim H^1( 3 {\cdot} 2^n \textrm{-gon}) = \colim \ZZ = \ZZ. \]

\end{proof}

Write $i_P(C, C') = \sum_{n \in \NN} \length \Tor_n^{\OO_{X, P}}(\OO_{C, P}, \OO_{C',P})$ for the intersection number of two curves $C, C'$ in a surface $X$ intersecting properly at a point $P$.

\begin{lemm} \label{lemm:three-curves}
Suppose that $C, C', C''$ are three irreducible curves in a smooth surface $X$ over a field $k$, such that $C' \cup C''$ is a snc divisor, $C \cap C' \cap C''$ contains a rational point $P$, and $C, C', C''$ are all smooth at $P$. Let $π: \Bl_PX = \widetilde{X} \to X$ be the blowup at $P$, let $\widetilde{C}, \widetilde{C}', \widetilde{C}''$ be the strict transforms, and $D = π^{-1}(P)$ the exceptional divisor.
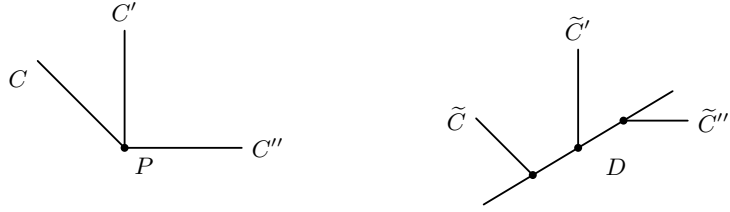
\begin{enumerate}
 \item If $i_P(C, C') = i_P(C, C'')$ then 
\begin{align*}
\widetilde{C} \cap D \cap \widetilde{C}' &= \varnothing, \\
\widetilde{C} \cap D \cap \widetilde{C}'' &= \varnothing.
\end{align*}

\[
\begin{tikzpicture}[
  line width=.7pt,
  line cap=round,
  >={Stealth[length=5pt]},
  every node/.style={font=\small}
]


\begin{scope}[shift={(0,4.4)}]
  \draw (0,0)--(0,1.55) node[above] {$C'$};
  \draw (0,0)--(1.55,0) node[right] {$C''$};
  \draw (0,0)--(-1.15,1.15)
    node[below left] {$C$};

  \fill (0,0) circle (1.5pt)
    node[below right] {$P$};
\end{scope}

\begin{scope}[shift={(6,4.4)}]
  \draw (-1.25,-.75)--(1.25,.75)
    node at (0.5,-0.25) {$D$};

  \draw (-.60,-.36)--(-1.35,.39)
    node[left] {$\widetilde C$};
  \draw (0,0)--(0,1.30)
    node[above] {$\widetilde C'$};
  \draw (.60,.36)--(1.45,.36)
    node[right] {$\widetilde C''$};

  \fill (-.60,-.36) circle (1.5pt);
  \fill (0,0) circle (1.5pt);
  \fill (.60,.36) circle (1.5pt);
\end{scope}
\end{tikzpicture}
\]

 \item If $i_P(C, C') > i_P(C, C'')$ then 
\begin{align*}
i_{D\cap\widetilde C'}(\widetilde C,D)
&= 
i_P(C,C'') \\
i_{D\cap\widetilde C'}(\widetilde C,\widetilde C')
&= 
i_P(C,C')-i_P(C,C'') \\
\widetilde C\cap D\cap\widetilde C'' &= \varnothing.
\end{align*}
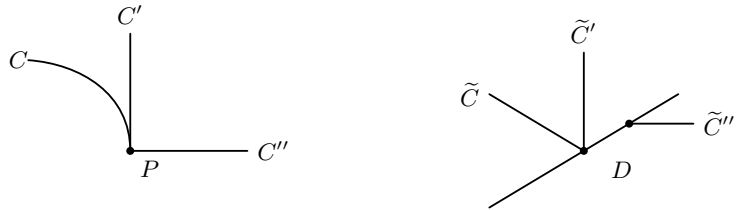
\[
\begin{tikzpicture}[
  line width=.7pt,
  line cap=round,
  >={Stealth[length=5pt]},
  every node/.style={font=\small}
]


\begin{scope}[shift={(0,.5)}]
  \draw (0,0)--(0,1.55) node[above] {$C'$};
  \draw (0,0)--(1.55,0) node[right] {$C''$};

  \draw (-1.35,1.20)
    .. controls (-.65,1.15) and (0,.75) ..
    (0,0);
  \node at (-1.48,1.20) {$C$};

  \fill (0,0) circle (1.5pt) node[below right] {$P$};
\end{scope}


\begin{scope}[shift={(6,.5)}]
  \draw (-1.25,-.75)--(1.25,.75)
    node at (0.5,-0.25) {$D$};

  \draw (0,0)--(-1.25,.75)
    node[left] {$\widetilde C$};
  \draw (0,0)--(0,1.30)
    node[above] {$\widetilde C'$};
  \draw (.60,.36)--(1.45,.36)
    node[right] {$\widetilde C''$};

  \fill (0,0) circle (1.5pt);
  \fill (.60,.36) circle (1.5pt);
\end{scope}
\end{tikzpicture}
\]
\end{enumerate}
\end{lemm}

\begin{proof}
Note that in a smooth surface all curves are lci, so the intersection multiplicity is the length of the local ring. That is, only $\Tor_0$ is nonzero. Since $C$ is smooth at $P$, we have $\widehat{\OO}_{C, P} \cong k[[t]]$ and  the intersection multiplicities are therefore obtained as 
\[ x' = u't^{i(C, C')} \qquad \textrm{ and } \qquad x'' = u''t^{i(C, C'')} \]
where $u', u'' \in k[[t]]^*$ and $x', x'' \in \widehat{\OO}_{X, P}$ are local coordinates defining $C', C''$ at $P$. 

Blowing up, we get charts with coordinates $x', \tfrac{x''}{x'}$ and $\tfrac{x'}{x''}, x''$ respectively. On these new charts, locally
near $\widetilde{C}' \cap D$ and $\widetilde{C}'' \cap D$ respectively, the curves 
 $\widetilde{C}', \widetilde{C}'', D$ are defined by $1$, $\tfrac{x''}{x'}$, $x'$ and $\tfrac{x'}{x''}, 1, x''$ respectively. In particular, if $i(C, C') = i(C, C'')$ then $\Spec k[[t]] \to \Spec k[[x', x'']]$ factors through both charts, and pulling the coordinates of $\widetilde{C}', \widetilde{C}''$ back to $k[[t]]$ gives units $1, \tfrac{u''}{u'}$ and $\tfrac{u'}{u''}, 1$ on both charts. 

Suppose $i(C, C') > i(C, C'')$. Since $t^{i(C,C'') - i(C,C')} \notin k[[t]]$ the scheme morphism $\Spec k[[t]] \to \Spec k[[x', x'']]$ does not factor through the chart $\Spec k[[x', x'']][\tfrac{x''}{x'}]$. On the chart $\Spec k[[x', x'']][\tfrac{x'}{x''}]$ the local equations for $\widetilde{C}'$ and $D$ are 
$\tfrac{x'}{x''} = \tfrac{u'}{u''} t^{i(C,C') - i(C,C'')}$
 and 
$x'' = u'' t^{i(C,C'')}$
so 
$i(\widetilde{C}, \widetilde{C}') = i(C,C') - i(C,C'')$ 
and 
$i(\widetilde{C}, D) = i(C,C'')$.
\end{proof}

\section{A construction for $d \geq 2$}

\begin{prop}
Suppose $d \geq 2$. Then there exists a ring $R$ with
\[ \Kdim R = 1, \qquad \textrm{ and } \qquad K_{-d}(R) \neq 0. \]
\end{prop}

\begin{proof}
Let $X_0$ be a toric $d$-fold such as $X_0 = \PP^d$ with 
cocharacter lattice $N = \hom(\GG_m, T)$ and fan $Σ_0$. 
Consider the poset $\Fan$ of finite completerational fans and its subposet $\Fan^{sm,proj}$ of fans $\Sigma$ such that $X_\Sigma$ is smooth and projective. The poset relation is subdivision. %
Note that $\Fan$ is filtered and $\Fan^{sm,proj}$ is cofinal. Indeed,
\begin{enumerate}
 \item any two finite complete rational fans $\Delta, \Delta'$ admit a common subdivision,%
\footnote{For example take $\{σ \cap σ' \mid σ \in \Delta, σ' \in \Delta'\}$} %
and
 \item any complete fan $\Sigma$ admits a subdivision $\Sigma'$ such that $X_{\Sigma'}$ is nonsingular and projective, \cite[Thm.6.1.18, Thm.11.1.9]{CLS11}.
\end{enumerate}
For each $Σ \in \Fan^{sm,proj}$ let
\[ 
R_Σ = \bigcap_{x \in X_Σ^T} \OO_{X_Σ, x} \subseteq k(X_Σ)
\] 
be the semilocalisation of $X_Σ$ at the fixed points $X_Σ^T$ and $I_Σ \subseteq R_Σ$ the ideal in the cartesian square
\[ \xymatrix{
\Spec R_Σ/I_Σ \ar[r] \ar[d] & \partial X_Σ = X_Σ \setminus T \ar[d] \\
\Spec R_Σ \ar[r] & X_Σ 
} \]
where $\partial X_Σ = X_Σ \setminus T$ is the closed complement of the dense orbit $T \subseteq X_Σ$. Note that $\Spec R_\Sigma \to X_\Sigma$ is pro-open, thanks to our hypothesis that $X_\Sigma$ is projective.%
\footnote{Choose $X_\Sigma \subseteq \PP^N$ for some $N$ and a hyperplane missing the finitely many fixed points of $X_\Sigma$.} %

Using the resulting filtered system we define
\begin{align*}
R_∞ &= \colim R_Σ \\
I_∞ &= \colim I_Σ \\
R &= R_∞ \times_{R_∞ / I_∞} k
\end{align*}
It remains to show that $\Kdim R = 1$ and $K_{-d}(R) \neq 0$.

\emph{The Krull dimension of $R$ is one.} Since the topological space of the pushout scheme $\Spec R_∞ \sqcup_{\Spec (R_∞ / I_∞)} \Spec k$ is the pushout topological space, it suffices to show that every nongeneric point $y \in  \Spec R_∞$ has $y \in \Spec R_∞ / I_∞$. Equivalently, we want $y_Σ \in \partial X_Σ = X_Σ \setminus T$ for all $Σ$ where $y_Σ$ is the image of $y$ in $X_Σ$. 

Suppose not. Then we have a filtered system of proper subvarieties $\{Y_Σ \subseteq X_Σ\}_{Σ \in \Fan^{sm,proj}}$ such that $Y_Σ \cap X_Σ^T \neq \varnothing$ and $Y_Σ \cap T \neq \varnothing$ for all $Σ$.%
\footnote{Indeed if $Y_Σ \cap X_Σ^T = \varnothing$ then the generic point of $Y_Σ$ is not in the localistion $R_Σ$ and if $Y_Σ \cap T = \varnothing$ then $Y_Σ \subseteq \partial X_Σ = X_Σ \setminus T$.} %
This contradicts the following lemma.

\begin{lemm}
For every proper irreducible (not necessarily toric) subvariety $Y_0 \subseteq X_0$ with $Y_0 \cap T \neq \varnothing$, there exists some $Σ \in \Fan^{sm,proj}$ such that the strict transform $Y_Σ \subseteq X_Σ$ does not contain any fixed points of the toric variety $X_Σ$.
\end{lemm}

\begin{proof}
The fixed points of $X_Σ$ correspond to maximal cones $σ \in Σ$, \cite[Thm.3.2.6]{CLS11}%
. On the other hand, the subvariety $Y^\circ_Σ := Y_Σ \cap T \subseteq T$ has an associated tropical variety $\trop(Y^\circ_Σ) \subseteq N_\RR$, \cite[Def.3.2.1]{MS15}%
, and we have
\[ O(σ) \in Y_Σ  \qquad \iff \qquad \trop(Y^\circ_Σ) \cap \relint(σ) \neq \varnothing, \]
 \cite[Thm.6.3.4]{MS15}%
. Here $\relint(σ)$ is the is the interior of $σ$ inside its affine span, \cite[p.60]{MS15}%
. The tropical variety $\trop(Y^\circ_Σ)$ is the support of a finite%
\footnote{
\cite[Prop.3.2.8]{MS15}%
 realizes $\trop(Y^\circ_Σ)$ as the support of a subcomplex of the Gröbner complex $Σ(I_{proj})$, which is finite by \cite[Cor.2.5.11]{MS15}%
.
} %
 rational fan of pure dimension $\dim Y^\circ_Σ$, \cite[Thm.3.3.5]{MS15}%
.%
\footnote{See also p.307, ``We give the field $K$ the trivial valuation, so $\trop(Y)$ can be given the structure of a polyhedral fan.'' \cite[p.307]{MS15}%
} %
Let $\Delta$ be this rational fan of $\trop(Y^\circ_Σ)$. %
Let $\Sigma_0'$ be a subdivision of $\Sigma_0$ containing a subdivision of $\Delta$, and choose a subdivision $\Sigma$ of $\Sigma_0'$ such that $X_\Sigma$ is smooth and projective. Since $Y_0$ has dimension $< d$, the maximal cones in $\Delta$ have dimension $< d$, so they all have trivial intersection with $\relint(σ)$ for maximal cones $σ$ of $\Sigma$. Hence, $O(σ) \notin Y_\Sigma$ for all maximal $σ$.
\end{proof}

\emph{The $K$-group $K_{-d}(R)$ is nonzero.} 
Consider the exact sequence 
\[ 
\underbrace{K_{1-d}(R_∞) \oplus K_{1-d}(k)}_{=0} \to 
K_{1-d}(R_∞ / I_∞) \to 
K_{-d}(R) \to 
\underbrace{K_{-d}(R_∞) \oplus K_{-d}(k)}_{=0} 
\]
Here the vanishing 
results from the fact that $R_∞$ is a filtered colimit of semilocalisations of smooth varieties, and regular rings have no negative $K$-theory. We claim that $K_{1-d}(R_∞ / I_∞) \cong \ZZ$. It suffices to show each $K_{1-d}(R_n / I_n) \cong \ZZ$ in a way compatible with the transition morphisms. Since the Krull dimension of the noetherian rings $R_n/I_n$ is $d-1$ we have $K_{1-d}(R_n / I_n) \cong H_{cdh}^{d-1}(R_n / I_n, \ZZ)$, \cite[Cor.D]{KST18}%
.%
\footnote{To show the desired nonvanishing, we could just as well use $K_{1-d}(R_n / I_n)[\tfrac{1}{p}] \cong KH_{1-d}(R_n / I_n)[\tfrac{1}{p}]$ in which case the isomorphism follows directly from the cdh descent spectral sequence and cohomological dimension of the cdh topology.
} %
Repeatedly applying closed Mayer-Vietoris one sees that $H_{cdh}^*(R_n / I_n, \ZZ)$ is the \v{C}ech cohomology of the closed covering of $\Spec R_n/I_n$ by irreducible components. In dimension $d-1$ this is
\[ H_{cdh}^{d-1}(R_n/I_n) = \coker \left ( \bigoplus_{e \in \cE} \ZZ \to  \bigoplus_{v \in \cV} \ZZ \right ) \]
where $\cV = X_n^T$ is the set of orbits of dimension zero and $\cE$ is the set of orbits of dimension one. Hence, this is the $H_0$ of the incidence graph $\cG$ whose vertices are the $\dim = d$ cones of $Σ_n$ and edges are $\dim = d-1$ cones. This graph has one connected component so $H_0 = \ZZ$. 
%
%
%
%
%
%
Connectivity does not change under subdivision. So we have
\begin{align*}
K_{-d}(R) 
= K_{1-d}(R_∞/I_∞) 
&= \colim K_{1-d}(R_n/I_n) 
\\&= \colim H^{d-1}_{cdh}(R_n/I_n, \ZZ) 
= \colim H_0(\cG_n, \ZZ)
&= \colim \ZZ
\\&&=\ZZ
\end{align*}
\end{proof}

\bibliographystyle{alpha}
\bibliography{bib}

\end{document}